\documentclass[12pt]{amsart}

\usepackage{amsfonts,amsmath,amsthm,amssymb,latexsym,amsthm,newlfont,enumerate,color}

\newif\ifslide

\theoremstyle{plain}

\ifdefined\spacer
\newtheorem{theorem}{Theorem}
 \else
\newtheorem{theorem}{Theorem}[section]
\fi

\newtheorem{corollary}[theorem]{Corollary}
\newtheorem{lemma}[theorem]{Lemma}

\newtheorem{proposition}[theorem]{Proposition}
\newtheorem{definition-lemma}[theorem]{Definition-Lemma}

\newtheorem{red-question}[theorem]{\textcolor{red}{Question}}

\theoremstyle{definition}
\newtheorem{definition}[theorem]{Definition}
\newtheorem{remark}[theorem]{Remark}

\newtheorem{example}[theorem]{Example}

\def\ideal#1.{I_{#1}}
\def\ring#1.{\mathcal {O}_{#1}}

\def\fring#1.{\hat{\mathcal {O}}_{#1}}
\def\proj#1.{\mathbb {P}(#1)}
\def\pr #1.{\mathbb {P}^{#1}}
\def\dpr #1.{\hat{\mathbb {P}}^{#1}}
\def\af #1.{\mathbb A^{#1}}
\def\Hz #1.{\mathbb F_{#1}}
\def\Hbz #1.{\overline{\mathbb F}_{#1}}
\def\fb#1.{\underset #1 {\times}}
\def\rest#1.{\underset {\ \ring #1.} \to \otimes}
\def\au#1.{\operatorname {Aut}\,(#1)}
\def\deg#1.{\operatorname {deg } (#1)}

\def\pic#1.{\operatorname {Pic}\,(#1)}
\def\pico#1.{\operatorname{Pic}^0(#1)}
\def\picg#1.{\operatorname {Pic}^G(#1)}
\def\ner#1.{NS (#1)}
\def\rdown#1.{\llcorner#1\lrcorner}
\def\rfdown#1.{\lfloor{#1}\rfloor}
\def\rup#1.{\ulcorner{#1}\urcorner}
\def\rcup#1.{\lceil{#1}\rceil}

\def\n1#1.{\operatorname {N_1}(#1)}  
\def\cn1#1.{\overline{\operatorname {N^1}(#1)}} 
\def\cone#1.{\operatorname {NE}(#1)}     
\def\ccone#1.{\overline{\operatorname {NE}}(#1)}
\def\none#1.{\operatorname {NF}(#1)}
\def\cnone#1.{\overline{\operatorname {NF}}(#1)}
\def\mone#1.{\operatorname {NM}(#1)} 
\def\cmone#1.{\overline{\operatorname {NM}}(#1)}

\def\coef#1.{\frac{(#1-1)}{#1}}
\def\vit#1.{D_{\langle #1 \rangle}}
\def\mm#1.{\overline {M}_{0,#1}}
\def\H1#1.{H^1(#1,{\ring #1.})}
\def\ac#1.{\overline {\mathbb F}_{#1}}

\def\adj#1.{\frac {#1-1}{#1}}
\def\spn#1.{\overline{#1}}
\def\pek#1.#2.{\Cal P^{#1}(#2)}
\def\plk#1.#2.{\Cal P^{\leq #1}(#2)}
\def\ev#1.{\operatorname{ev_{#1}}}
\def\ilist#1.{{#1}_1,{#1}_2,\dots}
\def\bminv#1.{(\nu_1,s_1;\nu_2,s_2;\dots ;\nu_{#1},s_{#1};\nu_{r+1})}
\def\zinv#1.{(\nu_1,s_1;\nu_2,s_2;\dots ;\nu_{#1},s_{#1};0)}
\def\iinv#1.{(\nu_1,s_1;\nu_2,s_2;\dots ;\nu_{#1},s_{#1};\infty)}

\def\scr #1.{\mathcal #1}

\def\llist#1.#2.{{#1}_1,{#1}_2,\dots,{#1}_{#2}}
\def\ulist#1.#2.{{#1}^1,{#1}^2,\dots,{#1}^{#2}}
\def\lomitlist#1.#2.{{#1}_1,{#1}_2,\dots,\hat {{#1}_i}, \dots, {#1}_{#2}}
\def\lomitlistz#1.#2.{{#1}_0,{#1}_1,\dots,\hat {{#1}_i}, \dots, {#1}_{#2}}
\def\loc#1.#2.{\Cal O_{#1,#2}}
\def\fderiv#1.#2.{\frac {\partial #1}{\partial #2}}
\def\deriv#1.#2.{\frac {d #1}{d #2}}

\def\map#1.#2.{#1 \longrightarrow #2}
\def\rmap#1.#2.{#1 \dasharrow #2}
\def\emb#1.#2.{#1 \hookrightarrow #2}
\def\non#1.#2.{\text {Spec }#1[\epsilon]/(\epsilon)^{#2}}
\def\Hi#1.#2.{\text {Hilb}^{#1}(#2)}
\def\sym#1.#2.{\operatorname {Sym}^{#1}(#2)}
\def\Hb#1.#2.{\text {Hilb}_{#1}(#2)}
\def\Hm#1.#2.{\Hom_{#1}(#2)}
\def\prd#1.#2.{{#1}_1\cdot {#1}_2\cdots {#1}_{#2}}
\def\Bl #1.#2.{\operatorname {Bl}_{#1}#2}
\def\pl #1.#2.{#1^{\otimes #2}}
\def\mgn#1.#2.{\overline {M}_{#1,#2}}
\def\ialist#1.#2.{{#1}_1 #2 {#1}_2, #2\dots}
\def\pair#1.#2.{\langle #1, #2\rangle}
\def\vandermonde#1.#2.{\left|
\begin{matrix}
1 & 1 & 1 & \dots & 1\\
{#1}_1 & {#1}_2 & {#1}_3 & \dots & {#1}_{#2}\\
{#1}_1^2 & {#1}_2^2 & {#1}_3^2 & \dots & {#1}_{#2}^2\\
\vdots & \vdots & \vdots & \ddots & \vdots\\
{#1}_1^{#2-1} & {#1}_2^{#2-1} & {#1}_2^{#2-1} & \dots & {#1}_{#2}^{#2-1}\\
\end{matrix}
\right|
}
\def\vandermondet#1.#2.{\left|
\begin{matrix}
1 & {#1}_1   & {#1}_1^2 & \dots & {#1}_1^{#2-1}\\
1 & {#1}_2   & {#1}_2^2 & \dots & {#1}_2^{#2-1}\\
1 & {#1}_3   & {#1}_3^2 & \dots & {#1}_3^{#2-1}\\
\vdots & \vdots & \vdots & \ddots & \vdots\\
1 & {#1}_{#2}& {#1}_{#2}^2 & \dots & {#1}_{#2}^{#2-1}\\
\end{matrix}
\right|
}
\def\gr#1.#2.{\mathbb{G}(#1,#2)}

\def\alist#1.#2.#3.{{#1}_1 #2 {#1}_2 #2\dots #2 {#1}_{#3}}
\def\zlist#1.#2.#3.{#1_0 #2 #1_1 #2\dots #2 #1_{#3}}
\def\lomitlist30#1.#2.#3.{{#1}_0,{#1}_1 #2 \dots #2\hat {{#1}_i} #2\dots #2 {#1}_{#3}}
\def\lmap#1.#2.#3.{#1 \overset{#2}{\longrightarrow} #3}
\def\mes#1.#2.#3.{#1 \longrightarrow #2 \longrightarrow #3}
\def\ses#1.#2.#3.{0\longrightarrow #1 \longrightarrow #2 \longrightarrow #3 \longrightarrow 0}
\def\les#1.#2.#3.{0\longrightarrow #1 \longrightarrow #2 \longrightarrow #3}
\def\res#1.#2.#3.{#1 \longrightarrow #2 \longrightarrow #3\longrightarrow 0}
\def\Hi#1.#2.#3.{\text {Hilb}^{#1}_{#2}(#3)}
\def\ten#1.#2.#3.{#1\underset {#2}{\otimes} #3}
\def\lomitlist30#1.#2.#3.{{#1}_0 #2 {#1}_1 #2 \dots #2 \hat {{#1}_i} #2 \dots #2 {#1}_{#3}}
\def\mderiv#1.#2.#3.{\frac {d^{#3} #1}{d #2^{#3}}}

\def\Hom{\operatorname{Hom}}

\def\dim{\operatorname{dim}}

\def\deg{\operatorname{deg}}

\def\dep{\operatorname{dep}}
\def\aw{\operatorname{aw}}

\def\GL{\operatorname{GL}}

\def\mult{\operatorname{mult}}

\def\fix{\operatorname{Fix}}
\def\bfix{\operatorname{\mathbf{Fix}}}

\def\rest{\operatorname{res}}

\def\e{\Cal E}

\def\e1{E_1}
\def\e2{E_2}

\def\mapdown#1{\big\downarrow\rlap{$\vcenter{\hbox{$\scriptstyle#1$}}$}}

\def\mapse#1{
{\vcenter{\hbox{$\mathop{\smash{\raise1pt\hbox{$\diagdown$}\!\lower7pt
\hbox{$\searrow$}}\vphantom{p}}\limits_{#1}\vphantom{\mapdown{}}$}}}}

\def\VR#1.{height#1pt&\omit&&\omit&&\omit&&\omit&&\omit&\cr}

\def\VRT#1.{height#1pt&\omit&&\omit&\cr}

\title{Effective finite  generation for adjoint rings}

\author{Paolo Cascini}
\address{Department of Mathematics\\
Imperial College London\\
180 Queen's Gate\\
London SW7 2AZ, UK}
\email{p.cascini@imperial.ac.uk}

\author{De-Qi Zhang}
\address{Department of Mathematics\\ National University of Singapore\\ 2 Science Drive 2 \\ Singapore 117543, Singapore}\email{matzdq@nus.edu.sg}
\thanks{The first author was partially supported by an EPSRC Grant and the second author was supported by an ARF of NUS.
Some of the work was completed while the  first author was visiting the Institute of Mathematics of the Romanian Academy. He would like to thank  F. Ambro for his generous hospitality. We would like to thank F. Ambro, A. Corti, A.S. Kaloghiros and V. Lazi\'c for several very  useful discussions.}

\begin{document}

\begin{abstract}
We describe a bound on the degree of the generators for some adjoint rings on surfaces and threefolds. 
\end{abstract}

\maketitle
\tableofcontents

\section{Introduction}

The aim of this paper is to provide a first step towards an effective version of the finite generation of adjoint rings.

The study of effective results in birational geometry has a long history. 
On the one hand, the boundedness   results on the pluricanonical maps for varieties of general type, due to Hacon, M\textsuperscript cKernan, Takayama and Tsuji \cite{HM06,Takayama06}, laid the foundation for a great deal of work towards an effective version of the (log)-Iitaka fibrations (e.g. \cite{VZ09,TX08}). 
On the other hand, Koll\'ar's   effective version  \cite{Kollar93b}   of the Kawamata-Shokurov's base point free theorem provides  an explicit  bound for a multiple of a nef adjoint divisor to be base point free. 
Finally, many recent results on the geography of projective threefolds of general type yield, in particular, a description of the singularities which may appear on varieties of general type (e.g. see \cite{ChenChen10a,ChenChen10}). 

The main goal of this paper is to combine together some of these results and study an effective version of the finite generation for adjoint rings. More specifically, given a Kawamata log terminal pair $(X,B)$, the main result of \cite{BCHM10} implies that the canonical ring $R(X,K_X+B)$ of $K_X+B$ is finitely generated (see also \cite{Siu08,CL10a}). Moreover, if $A$ is an ample $\mathbb Q$-divisor and $B_1,\dots,B_k$ are $\mathbb Q$-divisors such that $(X,B_i)$ is Kawamata log terminal for all $i=1,\dots,k$, then the associated adjoint ring $R(X;K_X+A+B_1,\dots,K_X+A+B_k)$ is also finitely generated  (cf. Definition \ref{d_adjointrings}).
On the other hand, the problem about the finite generation of $R(X;K_X+B_1,\dots,K_X+B_k)$, without the assumption of $B_i$ being big, is still open and it implies the abundance conjecture (e.g. see \cite{CL10}). Thus, it is reasonable to ask if there exists a bound  on the number of generators of the adjoint ring on a smooth projective variety which depends only on the numerical and topological invariants associated to the pairs $(X,B_i)$ for $i=1,\dots,k$. 

\medskip 

Inspired by these questions, our main result is the following:

\begin{theorem}\label{t_ld_threefold}
Let $(X,B)$ be a  Kawamata log terminal projective threefold such that $X$ is smooth. Assume that $B$ is nef and that  $B$ 
or $K_X+B$ is big. 
 Let $a$ be a positive integer such that $aB$ is Cartier. 

Then, there exists a positive integer $q$, depending only on  the Picard number $\rho(X)$ of $X$  such that  $R(X,qa(K_X+B))$ is generated in degree $5$ (cf. Definition \ref{d_adjointrings}).
\end{theorem}

As a direct consequence, we obtain: 

\begin{corollary}\label{c_main}
Let $X$ be a smooth projective threefold of general type.

 Then there exists a constant $m$ which depends only on the Picard number $\rho(X)$ such that the stable base locus of $K_X$ coincides with the base locus of the linear system 
 $|mK_X|$. 
\end{corollary}

In addition, we obtain a stronger version of the theorem above in the case of surfaces. It is worth to mention that 
the proofs of these results rely in a crucial way on the classification of Kawamata log terminal surface singularities and terminal threefold singularities. 
\medskip

The paper is organized as follows. In section \ref{s_pr}, we describe the main tools used in the paper, which are mainly based on Koll\'ar's effective base point free theorem, Mumford's regularity theorem and the classification of surface and threefold singularities. In section \ref{s_bound}, we describe a bound on the index of the singularities of the minimal model of a smooth projective threefold $X$, which depends on the Picard number $\rho(X)$ of $X$.   The bound is obtained as a consequence of a recent result by Chen and Hacon \cite{ChenHacon11}. 
Finally, in section \ref{s_fg}, we prove the main results of the paper and we show, in many examples, that some of these results are optimal. In particular, the bound on the degree of the generators of a Kawamata log terminal pair $(X,B)$  depends on the Picard number of the projective variety $X$. 

Note that all the bounds obtained in the paper are easily computable, but they are far away from being sharp. For this reason, we  often omit an explicit description of these bounds.  

\section{Preliminary results}\label{s_pr}

\subsection{Notation}
We work over the field of complex numbers $\mathbb C$. 
We refer to \cite{KM98} for the classical  definitions of singularities in the Minimal Model Program.  In particular, given a log pair $(X,B)$, we denote by $a(\nu,X,B)$ the {\em discrepancy} of $(X,B)$ with respect to a valuation $\nu$.  
A rational map $f\colon \rmap X.Y.$ between normal projective varieties $X$ and $Y$ is a \emph{contraction} if the inverse map $f^{-1}$ does not contract any divisors.
The {\em exceptional locus} of $f$ is the subset of $X$ on which $f$ is not an isomorphism.  

Let $f\colon \rmap X.Y.$ be a proper birational contraction of normal projective varieties and let $D$ be an $\mathbb R$-Cartier divisor on $X$ such that $D_Y=f_*D$ is also $\mathbb R$-Cartier. Then $f$ is 
$D$-\emph{non-positive} (respectively $D$-\emph{negative}) if for some common resolution $p\colon \map W.X.$ and $q\colon \map W.Y.$ which resolves the indeterminacy locus of $f$, we may write
$$p^*D=q^*D_Y+E$$
where $E\ge 0$ is $q$-exceptional (respectively $E\ge 0$ is $q$-exceptional and the support of $E$ contains the strict transform of the exceptional divisor of $f$).     
In particular, if $(X,B)$ is a Kawamata log terminal pair, $D=K_X+B$ and $D_Y=K_Y+B_Y$, then $f$  is $(K_X+B)$-non-positive (respectively $(K_X+B)$-negative) if and only if  
$$a(F,X,B)\le a(F,Y,f_*B) \qquad (\text{ respectively }a(F,X,B)<a(F,Y,f_*B) ~)$$ 
for all the prime divisors $F$ which are exceptional over $Y$.

 Let $(X,B)$ be a Kawamata log terminal pair. 
 A proper birational contraction  $f\colon \rmap X.Y.$ of normal projective varieties is a \emph{log terminal model} 
for  $(X,B)$ if $f$ is $(K_X+B)$-negative,
$Y$ is $\mathbb Q$-factorial and $K_Y+f_*B$ is nef. 
If $B=0$ then a log terminal model of $(X,B)$ is called a \emph{minimal model} of $X$.
We denote by LMMP$_n$ the classical conjectures in the Log Minimal Model program in dimension $n$. 
In particular, LMMP$_n$ implies that each Kawamata log terminal pair $(X,B)$ such that $K_X+B$ is pseudo-effective admits a log terminal model.

\begin{definition}\label{d_z}
Let $X$ be a smooth projective variety and let $D$ be a $\mathbb Q$-divisor on $X$. We denote by $\kappa(X,D)$ the Kodaira dimension of $D$. For any positive integer $q$ such that $qD$ is Cartier and the linear system $|qD|$ is not empty, we denote by $\fix|qD|$, the fixed part of the linear system $|qD|$.  Thus, if $\kappa(X,D)\ge 0$, we may define
$$\bfix(D)=\liminf_{q\to \infty} \frac 1 q \fix|qD|,$$
where the limit is taken over all  sufficiently divisible positive integers.
\end{definition}

\begin{remark}\label{r_z}
Let $(X,B)$ be a log smooth projective pair of dimension $n$ such that $\rfdown B.=0$ and let $f\colon \rmap X.Y.$ be a log terminal model of $K_X+B$. Then, if $B_Y=f_*B$, we may write 
$$K_X+B=f^*(K_Y+B_Y)+E$$ 
for some $f$-exceptional $\mathbb Q$-divisor $E\ge 0$. 
The negativity lemma (e.g. \cite[Lemma 3.6.2]{BCHM10}) implies that $E$ does not depend on the log terminal model $f$. In addition, assuming LMMP$_n$, it is easy to check that 
$E=\bfix(K_X+B)$.
\end{remark}

\begin{lemma}\label{l_nef}
Let $(X,B)$ be a  $\mathbb Q$-factorial projective Kawamata log terminal pair. Assume that $B$ is nef, $K_X+B$ is pseudo-effective and that $B$ or $K_X+B$ is big. 

Then there exists a sequence of steps 
$$X=\rmap X_0.\rmap .\dots. .X_k=Y.$$
of the $K_X$-minimal model program such that the induced birational map $f\colon \rmap X.Y.$ is a 
log terminal model of $(X,B)$. 
\end{lemma}
\begin{proof}
Let $A\ge 0$ be an ample $\mathbb Q$-divisor.
For any rational number $\varepsilon>0$, since $B$ is nef, we have that  $B+\varepsilon A$ is ample. Thus, there exist a $\mathbb Q$-divisor $H_\varepsilon\sim_{\mathbb Q}B+\varepsilon A$ and $\lambda_\varepsilon>0$ such that, for any sufficiently small $\varepsilon$, the pair $(X,\lambda_\varepsilon H_\varepsilon)$ is Kawamata log terminal and $K_X+\lambda_\varepsilon H_\varepsilon$ is nef.

 Let us consider the $K_X$-minimal model program of $X$ with scaling of $H_\varepsilon$ \cite[Remark 3.10.10]{BCHM10}. 
 
Note that if $B$ is not big, then by assumption, $K_X+B$ is big and therefore there 
exist  $\delta,\eta>0$ such that  $\eta(K_X+B)\sim_{\mathbb Q}\delta A+B'$ for some  $\mathbb Q$-divisor $B'\ge 0$ such that $(X,B+B')$ is Kawamata log terminal. Let $C=B+B'+(\delta +\varepsilon(1+\eta))A$. We may assume that $(X,C)$ is Kawamata log terminal. If $1\le t\le \lambda_\varepsilon$, we have
$$\begin{aligned}
(1+\eta)&(K_X+tH_\varepsilon)\sim_{\mathbb Q}  (1+\eta)(K_X+B+\varepsilon A+(t-1)H_\varepsilon)\\
&\sim_{\mathbb Q} K_X+B+B'+(\delta +\varepsilon(1+\eta))A+(t-1)(1+\eta)H_\varepsilon\\
&\sim_{\mathbb Q}K_X+C+(t-1)(1+\eta)H_\varepsilon.
\end{aligned}
$$
Thus, if $1\le t\le \lambda_\varepsilon$, a log terminal model of $(X,tH_\varepsilon)$ is also a log terminal of $(X,C+(t-1)(1+\eta)H_\varepsilon)$ and the $K_X-$minimal model with scaling of $H_\varepsilon$ coincides with the $(K_X+C)$-minimal model with scaling of $(1+\eta)H_\varepsilon$. 
 
 Therefore, after a finite number  steps of the $K_X$-minimal model program, we obtain a $K_X$-negative map $f_\varepsilon\colon \rmap X. X_\varepsilon.$ such that $X_\varepsilon$ is $\mathbb Q$-factorial and  $K_{X_\varepsilon}+f_{\varepsilon *} H_\varepsilon$ is nef. By finiteness of models \cite[Theorem E]{BCHM10}, there exists a sequence $\varepsilon_i$ such that $\lim \varepsilon_i=0$ and $X_{\varepsilon_i}$ is constant. In particular $K_{X_{\varepsilon_i}}+f_{{\varepsilon}_i *}B$ is nef.  Note that since  $f_{\varepsilon_i}\colon \rmap X.X_{\varepsilon_i}.$ is $(K_X+H_{\varepsilon_i})$-negative and $A$ is ample, it is also $(K_X+B)$-negative. Thus $f_{\varepsilon_i}$ is  a log terminal model of $(X,B)$.  
\end{proof}

\subsection{Koll\'ar's effective base point freeness}
In this section we  describe some easy generalisations of Koll\'ar's base point freeness theorem and Mumford's regularity theorem.

\begin{theorem}\label{t_ebpf}
Let $(X,B)$ be a projective Kawamata log terminal pair of dimension $n$ such that $K_X+B$ is nef and  $B$ or $K_X+B$ is big. Let $a$ be a positive integer such that $a(K_X+B)$ is Cartier. 

Then there exists a positive integer $q$ depending only on $n$ such that 
the linear system $|qa(K_X+B)|$ is base point free. 
\end{theorem}
\begin{proof}
If $B$ is big, then there exists an ample $\mathbb Q$-divisor $A$ and an effective $\mathbb Q$-divisor $D$ such that $B\sim_{\mathbb Q}A+D$. Then if $\varepsilon>0$ is a sufficiently small rational number and 
$\Delta=(1-\varepsilon)B+\varepsilon D$, then the pair $(X,\Delta)$ is Kawamata log terminal and 
$B\sim_{\mathbb Q}\Delta+\varepsilon A.$
Thus, if $M=a(K_X+B)$, then
$$M-(K_X+\Delta)\sim_{\mathbb Q} \varepsilon A + (a-1)(K_X+B)$$ 
is ample and the result follows by Koll\'ar's effective base point freeness theorem \cite{Kollar93b}.

Thus, we may assume that $K_X+B$ is big and nef. 
Let $L=2a(K_X+B)$. Then $L$ is Cartier and $L-(K_X+B)$ is big and nef. 
The result follows again from \cite{Kollar93b}. 
\end{proof}

\begin{remark}
Using the notation of Theorem \ref{t_ebpf}, by \cite[Theorem 1.1]{Kollar93b} 
we can take $q=4(n+2)!(n+1)$.
\end{remark}

\begin{lemma}\label{l_ebpf}
Let $(X,B)$ be a projective Kawamata log terminal surface  such that $K_X+B$ is nef. Let $a$ be a positive integer such that $a(K_X+B)$ is Cartier. 

Then there exists a positive integer  $m$ depending only on $a$ such that 
$|m(K_X+B)|$ is base point free.  
\end{lemma}
\begin{proof}
By Theorem \ref{t_ebpf}, we may assume that $K_X+B$ is not big. 
If $K_X+B\sim_{\mathbb Q} 0$, then the results follows from \cite[Theorem 3.1]{TX08}.

Thus, we may assume that there exists a map $f\colon \map X.C.$ onto a smooth curve $C$ such that $K_X+B=f^*D$ for some ample $\mathbb Q$-divisor $D$ on $C$.  We may assume that $D=K_C+B_C$ for some effective $\mathbb Q$-divisor $B_C$ on $C$ such that $\rfdown B_C.=0$ (e.g. see \cite{Kawamata98}). By  \cite[Theorem 8.1]{PS09},  there exists a constant $b$, depending only on $a$ such that 
$bB_C$ is Cartier. Thus, the results follows.  
\end{proof}

The next result follows closely the proof of Mumford's regularity theorem (e.g. see \cite[Theorem 1.8.3]{Lazarsfeld04a}).

\begin{proposition}\label{p_mumford}
Let $X$ be a normal projective variety of dimension $n$. Let $B_1,\dots,B_k$ be $\mathbb Q$-divisors on $X$ and let $a_1,\dots,a_k$ be positive integers such that $(X,B_i)$ is a Kawamata log terminal pair and there exist Cartier divisors $L_1,\dots,L_k$ such that $L_i\sim_{\mathbb Q}a_i(K_X+B_i)$ 
and the linear system $|L_i|$ is base point free, 
for $i=1,\dots,k$.   
Let $G=\sum_{i=1}^k b_i L_i$  for some positive integers $b_1,\dots,b_k$ and assume     that  $b_\ell > n+1$. 

Then,  the natural map 
$$H^0(X,\ring X.(G))\otimes H^0(X,\ring X.( L_\ell ))
\to H^0(X,\ring X.(G+ L_\ell ))$$
is surjective.
\end{proposition}
\begin{proof}
We first assume that $\sum_{i=1}^k L_i$ is not big. Since the linear system $|\sum_{i=1}^k L_i|$ is base point free, there exists a morphism with connected fibres $f\colon \map X.Y.$ onto a normal projective $Y$  such that $\sum_{i=1}^k L_i=f^*A$ for some very ample divisor $A$ on $Y$.   Since $L_1,\dots,L_k$ are nef, if $\xi$ is a curve contracted by $f$ then $L_i\cdot\xi=0$ for any $i=1,\dots,k$. 
In particular, since $|L_i|$ is base point free, and the restriction of $L_i$ to any fibre of $f$ is trivial, it follows that 
there exist Cartier divisors  $L'_1,\dots,L'_k$  such that $L_i=f^*L'_i$. By \cite[Theorem 4.1]{Ambro05}, it follows that $L'_i\sim_{\mathbb Q}a_i(K_{Y}+B'_i)$ for some $\mathbb Q$-divisor $B'_i$ such that 
$(X,B'_i)$ is Kawamata log terminal for any $i=1,\dots,k$. Note that the linear system $|L'_i|$ is base point free and that 
$H^0(X,\sum_{i=1}^k c_i L_i )\simeq H^0(Y,\sum_{i=1}^k c_iL'_i)$ for any non-negative integers $c_1,\dots,c_k$. Thus, after replacing $X$ by $Y$, $L_i$ by $L'_i$ and $B_i$ by $B'_i$, we may assume that $\sum_{i=i}^k L_i$ is big. 

Let $V=H^0(X,\ring X.( L_\ell ))$ and let $\mathcal V=V\otimes \ring X.$. Then
$$\mathcal V\otimes \ring X.(- L_\ell )\to \ring X.$$
is surjective. Thus, if $r=\dim V$, then  the sequence
$$\begin{aligned}
0=\wedge^{r+1}\mathcal V\otimes &\ring X.(-(r+1) L_\ell )\to \dots\\
&\to \wedge^2\mathcal V\otimes \ring X.(-2 L_\ell )\to \mathcal V\otimes \ring X.(- L_\ell )\to \ring X.\to 0
\end{aligned}
$$
is exact. Twisting by $\ring X.(G+ L_\ell )$ gives
$$\begin{aligned}
0\to & \wedge^{r}\mathcal V\otimes \ring X.(G-(r-1) L_\ell )\to \dots \\
&\to \wedge^2\mathcal V\otimes \ring X.(G-L_\ell )\to \mathcal V\otimes \ring X.(G)\to \ring X.(G+L_\ell )\to 0
\end{aligned}
$$

Since $\sum_{i=1}^kL_i$ is big, $b_\ell\ge n+2$ and $b_i\ge 1$ for $i\neq \ell$,  we have that
$$\begin{aligned}
\sum_{i\neq \ell} b_i L_i + (b_\ell - j) &L_\ell - (K_X+B_\ell)\\ 
&\sim_{\mathbb Q} 
\sum_{i\neq \ell} b_i L_i + (b_\ell - j) L_\ell -\frac 1 {a_\ell}  L_\ell \\
&\sim_{\mathbb Q}  \sum_{i=1}^k L_i + \sum_{i\neq \ell} (b_i-1)L_i+ \left (b_\ell - j - 1 - \frac 1 {a_\ell}\right )L_\ell
\end{aligned}
$$
is big and nef. Thus, Kawamata-Viehweg vanishing implies that
$$\begin{aligned}
H^j(X,\wedge^{j+1}\mathcal V &\otimes \ring X.(G-j L))=\wedge^{j+1}V\otimes H^j(X,\ring X.(G-j L_\ell ))\\
&=\wedge^{j+1}V\otimes H^j(X,\ring X.(\sum_{i\neq \ell} b_i L_i + (b_\ell - j) L_\ell  ))=0
\end{aligned}
$$
for any $j>0$.
Since 
$$H^0(X,\ring X.(G))\otimes H^0(X,\ring X.( L_\ell ))=H^0(X,\mathcal V\otimes \ring X.(G)),$$
the map
$$H^0(X,\ring X.(G))\otimes H^0(X,\ring X.( L_\ell ))\to H^0(X,\ring X.(G+ L_\ell ))$$
is surjective and the claim  follows. 
\end{proof}

\subsection{Adjoint Rings} In this section, we recall some basic notion about  adjoint rings.
 
\begin{definition}\label{d_adjointrings}
Let $X$ be a smooth projective variety and let $D_1,\dots,D_k$ be 
Cartier divisors on $X$. The \emph{adjoint ring associated to} $D_1,\dots,D_k$ is
$$R(X;D_1,\dots,D_k)=\bigoplus_{(a_1,\dots,a_k)\in \mathbb Z_{\ge 0}^k}H^0(X,\ring X.(\sum_{i=1}^k a_iD_i)).$$
We say that $R=R(X;D_1,\dots,D_k)$ is \emph{generated in degree} $m$ if $R$ is 
generated by sections of $H^0(X,\ring X.(\sum_{i=1}^k a_iD_i))$, for $a_1,\dots,a_k\in \{0,\dots,m\}$. 
\end{definition}

\begin{remark} The definition of adjoint rings can be easily extended to $\mathbb Q$-divisors $D_1,\dots,D_k$ on a smooth projective variety $X$, but it will not be used in this paper in this generality (e.g. see \cite{CL10a} for more details).
\end{remark}

\begin{remark}\label{r_b} 
Let $X$ be a smooth projective variety and let $B$ be a $\mathbb Q$-divisor on $X$ such that $(X,B)$ is Kawamata log terminal. Let $q$ be a positive integer such that $qB$ is Cartier. 
Then $R=R(X,q(K_X+B))$ is generated in degree $m$ if and only if the sections of 
$$\bigoplus_{a\le m}H^0(X,aq(K_X+B))$$  span $R$. 
\end{remark}

\begin{proposition}\label{p_b}
Let $X$ be a smooth projective variety of dimension $n$. Let $B_1,\dots,B_k$ be $\mathbb Q$-divisors on $X$ and let $a_1,\dots,a_k$ be positive integers such that $(X,B_i)$ is Kawamata log terminal and there exist
Cartier divisors $L_1,\dots,L_k$   such that  $L_i\sim_{\mathbb Q}a(K_X+B_i)$ and the linear system 
$|L_i|$ is base point free, for $i=1,\dots,k$. 

Then $R(X;L_1,\dots,L_k)$ is generated in degree $n+2$.
\end{proposition}

\begin{proof} 
Let $G=\sum_{i=1}^k m_i L_i$ for some integers $m_1,\dots,m_k\ge 0$ and assume that there exists $\ell\in\{1,\dots,k\}$ such that $m_\ell> n+2$. Then 
Proposition \ref{p_mumford} implies that 
$$H^0(X, \ring X.(G-L_\ell)\otimes H^0(X, \ring X.(L_\ell))\to H^0(X,\ring X.(G))$$
is surjective and the claim follows. 
\end{proof}

\begin{lemma}\label{l_b}
Let $X$ be a smooth projective variety and let $D$ be a Cartier divisor on $X$. Assume that $R(X,D)$ is generated in degree $m$ and let $q=m!$.

Then the stable base locus of $D$ is equal to the base locus of the linear system $|qD|$. 
In addition,  if $F=\bfix(D)$ (cf. Definition \ref{d_z}), then  $qF$ is a Cartier divisor. 
\end{lemma}

\begin{proof}
If $x$ is a point contained in the base locus of the linear system $|qD|$
and $m'\le m$ is a positive integer, then, since $m'$ divides $q$,  it follows that 
 any section of $H^0(X,\ring X.(m'D))$ vanishes at $x$. Thus, by assumption, any section of $H^0(X,\ring X.(\ell D))$ vanishes at $x$, for any positive integer $\ell$. In particular, $x$ is contained in the stable base locus of $D$ and the first claim follows. 
The proof of the second claim is analogous. 
\end{proof}

\subsection{Surface and threefolds singularities} 
In this section, we recall a few known facts about Kawamata log terminal singularities in dimension $2$ and terminal singularities in dimension $3$.

\begin{lemma}\label{l_fact}

Let $(X,B)$ be a log smooth surface such that  $\rfdown B.=0$ and $K_X+B$ is pseudo-effective. 
 Let $f\colon\map X.Y.$ be the log terminal model of $(X,B)$. 
 
 Then there exist birational morphisms $g\colon \map X.Z.$ and $h\colon \map Z.Y.$ which factorize $f$ and such that 
 \begin{enumerate}
 \item $g$ is a sequence of smooth blow-ups;
 \item $h$ contracts only divisors contained in the support of $g_*B$.
 \end{enumerate}
\end{lemma}

\begin{proof}
Let $h\colon \map Z .Y.$ be the minimal resolution of $Y$. Then, since $X$ is smooth, there exists a morphism $g\colon\map X.Z.$, which is a  
sequence of smooth blow-ups and such that $f=h\circ g$.  Let $C=g_*B$. Then, $h$ is the log terminal model of $(Z,C)$ and since $Z$ is the minimal resolution, it follows that $h$ does not contract any $(-1)$-curve. In particular, 
$a(F,Y)\le 0$ for any curve $F$ contracted by $h$. Thus, $a(F,Y,f_*B)\le 0$ and since $h$ is $(K_Z+C)$-negative, it follows that 
$a(F,Z,C)<0$. Therefore $F$ is contained in the support of $C$.  
\end{proof}

We proceed by bounding the index of a Kawamata log terminal surface with  respect to the graph of its minimal resolution:
\begin{proposition}\label{p_kltsurface} Let 
$(S,p)$ be the germ of a Kawamata log terminal surface and let $f\colon T\to S$ be the minimal resolution of $S$.
Assume that $E_1,\dots,E_k$ are the irreducible components of the exceptional divisor of $f$ and let $\varepsilon>0$ be  such that $a(E_i,S)\ge -1+\varepsilon$ for all $i=1,\dots,k$. 

Then there exists a constant $r=r(k,\varepsilon)$, depending only on $k$ and $\varepsilon$, such that 
$rD$ is Cartier for all Weil divisors $D$ on $S$. 
\end{proposition}
\begin{proof}
The germ of a Kawamata log terminal surface $(S,p)$ is given by the quotient of $\mathbb C^2$ by a finite subgroup $G$ of $\GL(2,\mathbb C)$, without quasi-reflections. Thus, it is enough to bound the order of $G$, depending on the graph associated to the minimal resolution of $(S,p)$. 
It follows from \cite[pag. 348]{Brieskorn68} that the order of $G$ is at most  $r=120k^2/\varepsilon^3$. 
\end{proof}

We now consider the germ $(X,p)$  of a terminal singularities in dimension $3$. The {\em index} of $X$ at  $p$ is the smallest positive integer $r=r(X,p)$ such that $rK_X$ is Cartier. In addition, it follows from the classification of terminal singularities \cite{Mori85}, that 
there exists a deformation of $(X,p)$ into a variety with $k\ge 1$ terminal singularities $p_1,\dots,p_k$ which are isolated cyclic quotient singularities of index $r(p_i)$. The set $\{p_1,\dots,p_k\}$ is called the {\em basket} $\mathcal B(X,p)$ of singularities of $X$ at $p$ \cite{Reid85}. The number $k$ is called the {\em axial weight} $\aw(X,p)$ of $(X,p)$. 
As in \cite{ChenHacon11}, we define 
$$\Xi(X,p)=\sum_{i=1}^k  r(p_i).$$
Thus, if $X$ is a projective variety of dimension $3$ and with terminal singularities, we may define
$$\Xi(X)=\sum_{p\in X} \Xi(X,p)\qquad \text{and}\qquad \aw(X)=\sum_{p\in X}\aw(X,p).$$ 
\begin{definition}
Let $X$ be a terminal projective variety of dimension $3$. Then a  $w$-resolution of $X$ is a sequence 

$$Y=\map X_N.\map .\dots. .X_0=X$$
such that 
\begin{enumerate}
\item $X_i$ has only terminal singularities for $i=1,\dots,N$ and $X_N$ has only Gorenstein singularities; 
\item the map $\map X_{i+1}.X_{i}.$ is a weighted blow-up at $P_i\in X_i$ with minimal discrepancy $1/m_i$ where $m_i$ is the index of $X_{i}$ at $P_i$, for $i=0,\dots,N-1$.
\end{enumerate}

\end{definition}

The following result is due to Hayakawa. 
\begin{theorem}
Let $X$ be a terminal projective threefold.

Then $X$ admits a $w$-resolution. \end{theorem}

\begin{proof}
See \cite[Theorem 6.1]{Hayakawa00}.
\end{proof}

Given $X$ as above, we define $\dep (X)$ the {\em depth}  of $X$ to be the minimum length of any $w$-resolution of $X$.

\begin{proposition}\label{t_ch}
Let $X,X'$ be terminal projective varieties of dimension $3$. Then
\begin{enumerate}
\item If $\rmap X.X'.$ is a  a flip, then $\dep (X)>\dep (X')$.

\item If $\map X. X'.$ is an extremal divisorial contraction to a curve, then
$\dep (X)\ge \dep (X')$. 
\item If $\map X.X'.$ is an extremal divisorial contraction to a point, then 
$\dep(X)\ge \dep(X')-1$. 
\end{enumerate}
\end{proposition}
\begin{proof}
See \cite[Proposition 2.15, 3.8, 3.9]{ChenHacon11}.
\end{proof}

\section{Bounding threefold terminal singularities}\label{s_bound}

The aim of this section is to give a bound on the singularities of a minimal projective threefold depending on the topology of its resolution. More specifically, we 
show that we can bound the sum of the  indices  of all the points in all the baskets  of $Y$  by a constant which depends only on the Picard number $\rho(X)$ of $X$.

We first show a bound on the number of flips for the minimal model program of a smooth projective threefold $X$:

\begin{lemma}\label{l_number}
Let $X$ be a smooth projective threefold. 

Then both the number of divisorial contractions and the number of flips in the minimal model program of $X$ are bounded by $\rho(X)$.  
In addition, if $Y$ is the minimal model of $X$ then the number of points of $Y$ with index greater than one is also bounded by $\rho(X)$. 
\end{lemma}

\begin{proof}
Clearly the number of divisorial contractions is bounded by $\rho(X)$. 
Let 
$$X=\rmap X_0.\rmap .\dots. .X_k=Y$$
be a  sequence of steps for the $K_X$-minimal model program of $X$, where $\rmap X.Y.$ is a minimal model of $X$. Let $d(X_i)$ be the Shokurov's difficulty of $X_i$, i.e.  
$$d(X_i)=\#\{\nu =\text{valuation}\mid a(\nu,X_i)<1,\quad \nu\text{ is exceptional over $X_i$}\}.$$
Then, since $X$ is smooth, we have $d(X)=0$. In addition, If $\rmap X_i.X_{i+1}$ is an extremal divisorial contraction, then 
$$d(X_i)\le d(X_{i-1})+1.$$
Finally, if $\rmap X_i.X_{i+1}$ is a flip, then 
$$d(X_i) \le d(X_{i-1})-1,$$
(e.g. see \cite{KM98}). Thus, the total number of flips is bounded by the number of divisorial contractions.

Finally, for each point $p\in Y$ such that $r(Y,p)>1$, the main result in \cite{Kawamata93} implies that  there exists a 
weighted blow-up $f\colon W\to Y$ of minimal discrepancy, i.e. if $E$ is the exceptional divisor of $f$ then 
$$a(E,Y)=\frac {1}{r(Y,p)}.$$
Thus, it follows that the number of such points is bounded by $d(Y)$. Thus, it is also bounded by $\rho(X)$ and the claim follows. 
\end{proof}

The following is a generalization of \cite[Proposition 2.13]{ChenHacon11}:
\begin{lemma}\label{l_xidep}
Let $Z$ be a terminal projective variety of dimension $3$. 

Then, 
$$\Xi(Z)\le 2 \dep (Z).$$
\end{lemma}

\begin{proof}We proceed by induction on $\dep(Z)$. If $\dep (Z)=0$ then $Z$ is Gorenstein and 
$\Xi(Z)=0$. Thus, the claim follows. 

Assume now that  $\dep (Z)>0$.  
We claim  that if $f\colon \map Y.Z.$ is a weighted blow-up of minimal discrepancy at the point $p\in Z$, then 
$$\Xi(Y)\ge\Xi(Z)-2.$$

We first prove the Lemma, assuming the claim.  Let $f\colon \map Y.Z.$ be the first weighted blow-up of minimal discrepancy in a $w$-resolution of $Z$. Then, by definition
we have that $\dep(Z)=\dep (Y)+1$ and by induction, we have that $\Xi(Y)\le 2\dep(Y)$. Thus
$$\Xi(Z)\le \Xi(Y)+2\le 2\dep (Y)+2=2\dep(Z),$$
and the lemma follows. 

The claim follows from the classification of weighted blow-ups in \cite{Hayakawa99,Hayakawa00}. 
For example, assume that $(Z,p)$ is a point of type $cA/r$ and $\map Y.Z.$ is a weighted blow-up of minimal discrepancy. Then, by the proof of \cite[Theorem 6.4]{Hayakawa99},  there exist positive integers $a,b$ satisfying $a+b=kr$ with $k\le \aw(Z)$ and such that the only points of index greater than $1$ on $Y$ are the following: a point $Q_1$ which is a cyclic quotient singularity of index $a$, a point $Q_2$ which is a cyclic quotient singularity of index $b$ and, if $k<\aw(Z,p)$, a point $Q_3$ which is of type $cA/r$ of index $r$ and axial weight $\aw(Z)-k$. Thus, since $\Xi(Z,p)=r\aw(Z,p)$, it follows that 
$$\Xi(Y)-\Xi(Z)=a+b+r(\aw(Z,p) -k) - r\aw(Z,p)=0.$$  
Similarly, if $(Z,p)$ is a point of of type $cAx/4$ and $\map Y. Z.$ is the weighted blow-up given by \cite[Theorem 7.4]{Hayakawa99}, then there exists a positive integer $k$  such that the only points of index greater than $1$ on $Y$ are the following: a point $Q_1$ which is cyclic of quotient singularity of index $2k+3$ and, if $\aw(Z,p)>k+1$, a point $Q_2$ which is of type $cD/2$ and such that $\aw(Y,Q_2)=\aw(Z,p)-k-1$. Thus, since $\Xi(Z,p)=2\aw(Z,p)+2$, it follows that
$$\Xi(Y)-\Xi(Z)= 2(\aw(Z,p)-k-1)+2k+3 - (2\aw(Z,p)+2)=-1.$$
It is easy to check that if $(Z,p)$ is a singularity of different type, then Hayakawa's list of weighted blow-ups of minimal discrepancy implies that  the inequality is satisfied in all the cases. 
\end{proof}

\begin{proposition}\label{p_2b2}
Let $X$ be a smooth projective threefold and 
assume that 
$$X=\rmap X_0.\rmap .\dots. .X_k=Y$$ 
is a sequence of steps for the $K_X$-minimal model program of $X$. 

Then 
$$\Xi(Y)\le 2\rho(X).$$
In particular, the inequality holds if $Y$ is the minimal model of $X$. 
\end{proposition}

\begin{proof}
Since $X$ is smooth, we have that $\dep(X)=0$.
By Proposition \ref{t_ch}, it follows that $\dep(X_i)\le \dep(X_{i-1})$ unless $\rmap X_{i+1}.X_{i}.$ is a divisorial contraction to a point and  in this case, we have $\dep(X_i)\le\dep(X_{i-1})+1$. Thus, $\dep(X_i)$ is bounded by the number of divisorial contractions in the first $i$ steps of the minimal model program of $X$. Thus, Lemma \ref{l_number}
implies that 
$\dep(Y)\le \rho(X)$. 

On the other hand, Lemma \ref{l_xidep} implies that 
$$\Xi(Y)\le 2\dep (Y).$$
Thus, the claim follows.
\end{proof}

\section{Effective Finite Generation}\label{s_fg}

We now proceed to show a bound on the degree of the generators of the adjoint ring associated to a Kawamata log terminal surface depending on the number of components of its boundary and the corresponding coefficients. 

\begin{proposition}\label{p_ld_surfaces}
Let $(X,B=\sum_{i=1}^p a_i S_i)$ be a log smooth projective surface, where $S_1,\dots,S_p$ are distinct prime divisors and
$\rfdown B.=0$. Let $a$ be a positive integer such that $aB$ is Cartier.

Then, there exists a positive integer $m=m(a,p)$, depending only on $a$ and $p$,  such that  $R(X,m(K_X+B))$ is generated in degree $4$. \end{proposition}
\begin{proof}
Let $f\colon \map X.S.$ be the log terminal model of $(X,B)$.  Lemma \ref{l_fact} implies that  $f$ factorizes through the minimal resolution $h\colon \map S'.S.$ of $S$. Let $g\colon \map X.S'.$ be the induced map. Then $h$ contracts only prime divisors which are rational curves with negative self-intersection and contained in the support of $g_*B$. 
Let $\Gamma$ be the strict transform of such a curve on $X$. Then 
$$\mult_\Gamma B \ge -a(\Gamma,S)$$
Thus, since $a\mult_\Gamma B$ is  a positive integer and $\mult_\Gamma B<1$, it follows that 
$$a(\Gamma,S)\ge -1 + \frac 1 a.$$

By  Proposition \ref{p_kltsurface}, it follows that there exists a constant $r=r(a,p)$ depending only on $a$ and $p$ such that $r D$ is Cartier for any Weil divisor $D$ on $S$. In particular, if  $B_S=f_*B$, then $ar(K_S+B_S)$ is a Cartier divisor. 
Thus, by Lemma \ref{l_ebpf} there exists a constant $q=q(a,p)$ depending on $a$ and $p$ such that 
the linear system $|q(K_S+B_S)|$ is  base point free. 
Let $L=q(K_S+B_S)$. Proposition \ref{p_mumford} implies that if $b>3$ then the natural map
$$H^0(S,\ring S.(bL ))\otimes H^0(S,\ring S.(L))\to H^0(S,\ring S.((b+1)L))$$
is surjective.
Thus, $R(X,q(K_X+B))$ is generated in degree $4$ and the claim follows.
\end{proof}

\begin{remark} Using the same notation as in Proposition \ref{p_ld_surfaces}, let 
 $q$ be the number of components of $B$ which are contained in the stable base locus of $K_X+B$. Then,
 the  same proof  shows that,  
$R(X,m(K_X+B))$ is generated in degree $4$ where $m$ is a constant depending on  $a$ and $q$.
Note that $q\le \rho(X)$. Thus, we can bound $m$ by a constant which depends on $a$ and the Picard number $\rho(X)$ of $X$. 
\end{remark}

The following example shows that the degree of the generators depends on the number of components $p$. 

\begin{example}
Let $r$ be a prime and let  $X_r$ be the smooth toric surface obtained by blowing-up a sequence of $r$ infinitesimal near points of $\mathbb P^2$. More specifically,  let $X_0=\mathbb P^2$ with a torus $T=(\mathbb C^*)^2$ acting on it, let 
 $X_1$ be the blow-up of $\mathbb P^2$ at a $T$-invariant point $p$ and for each $i=1,\dots,r$ let $X_i$ 
be the surface obtained by blowing-up the $T$-invariant point in the exceptional divisor of $f_{i-1}\colon \map X_{i-1}.X_{i-2}.$ which is not contained in the strict transform of the exceptional divisor of $f_{i-2}$. Then
$X_r$ admits a chain of $T$-invariant 
$(-2)$-curves  $e_1,\dots,e_{r-1}$ and a $T$-invariant $(-1)$-curve $e_r$, given the exceptional curve of $f_r$. Let $E=\sum_{i=1}
^re_i$, let $\pi\colon \map X_r.X_0.$ be the induced map and let $h=\pi^*\ell$ where  $\ell$ is the $T$-invariant line in $X_0$ not passing through $p$. 
Note that if $f\colon \map X_r.Y.$ is the map obtained by contracting the curves $e_1,\dots,e_{r-1}$, then  $Y$ admits a 
unique singular point of type $A_{r-1}$. 
Let 
$$G=\frac 1 r \sum_{i=1}^r i e_i.$$
Then the $T$-invariant curves on $X_r$ are contained in the classes $$e_1,\dots,e_r,h,h-rG,h-e_1.$$ 
It follows that if $a$ is a sufficiently large positive integer then $ah-2rG$ is nef and therefore the linear system $|a h-2rG|$ is base point free.
Thus, there exists a $\mathbb Q$-divisor $B_0\ge 0$ such that $B_0\sim_{\mathbb Q}  a h-2rG$, and 
if $B=B_0+\frac 1 2 \sum_{i=1}^r e_i$ then 
$(X,B)$ is  log smooth, $\rfdown B.=0$, and $2B$ is Cartier.  
It is easy to check that $f\colon \map X_r.Y.$ is the log terminal model of $(X,B)$ and that 
$$K_X+B=f^*(K_Y+f_*B)+E$$
where 
$$E=\frac 1 {2r}\sum_{i=1}^{r-1}  (r-i)e_i.$$
In particular, by Remark \ref{r_z}, we have $E=\bfix(K_X+B)$.
Thus, since $r$ is prime,  Lemma \ref{l_b} implies that if $a$ and $m$  are positive integer such that  $a$ is even and $R(X,a(K_X+B))$ is generated in degree $m$ then   either $a\ge r$ or $m\ge r$. 
\end{example}

The following example shows that Proposition \ref{p_ld_surfaces} cannot be extended to the log canonical (or dlt) case. 
\begin{example}
Let $X$ be the Hirzebruch surface $\mathbb F_r$ of prime degree $r\ge 3$, let $S$ be the unique curve of negative self-intersection $-r$ and let $H$ be a curve of self-intersection $r$. Then there exists an effective $\mathbb Q$-divisor  $B'\sim_{\mathbb Q} 2 H$ such that, $S$ is not contained in the support of $B'$ and  if 
$B=S+B'$ then $(X,B)$ is log smooth, $\rfdown B'.=0$ and  $2B'$ is Cartier. In particular,  the support of $B'$ does not intersect $S$. 

It is easy to check that  
$$\bfix(K_X+B)=\frac 2 rS.$$ 
Thus, Lemma \ref{l_b} implies that if $a$ is an even positive integer such that $R(X,a(K_X+B))$ is generated in degree $m$ then either $a\ge r$ or $m\ge r$.  
\end{example}

We now proceed with the proof of Theorem \ref{t_ld_threefold}:

\begin{proof}[Proof of Theorem \ref{t_ld_threefold}]
We may assume that $K_X+B$ is pseudo-effective, otherwise $R(X,K_X+B)$ is trivial. Since $B$ is big or $K_X+B$ is big, 
Lemma \ref{l_nef} implies that 
there exists a sequence of steps 
$$X=\rmap X_0.\rmap .\dots. .X_k=Y.$$
of the $K_X$-minimal model program such that the induced birational map $f\colon \rmap X.Y.$ is a log terminal model of $(X,B)$. Note that, in particular, since $X$ is smooth, $Y$ has terminal singularities. 

By  Proposition \ref{p_2b2}, it follows that there exists a constant $r$ depending only on $\rho(X)$  such that $r D$ is Cartier for any Weil divisor $D$ on $Y$. In particular, if  $B_Y=f_*B$, then $ra(K_Y+B_Y)$ is a Cartier divisor. By Theorem \ref{t_ebpf}  there exists a constant $q'$ 
such that, if $q=q'r$, then the linear system $|qa(K_Y+B_Y)|$ is  base point free. 
Similarly to Proposition  \ref{p_ld_surfaces}, Proposition \ref{p_mumford} implies that $R(X,qa(K_X+B))$ is generated in degree $5$, as claimed. 
\end{proof}

We expect that a more general result holds for any Kawamata log terminal pair threefold or even more in general for adjoint rings on threefolds, as in the case of surfaces. 
On the other hand, the following example shows that it is not possible to bound the degree of the generators of the canonical ring $R(X,K_X)$ of a smooth projective variety $X$ independently of $\rho(X)$, even assuming $B=0$.

\begin{example}
Let $r$ be  a prime number and let $Y_r$ be a projective variety of dimension $3$ such that $K_{Y_r}$ is ample and $Y_r$ admits a singular point of type
$$\frac 1 r(1,1,r-1).$$
Let $\map X_r.Y_r.$ be a resolution. Then there exists a prime divisor $E$ on $X_r$ which is  exceptional over $Y_r$ and such that $a(E,Y_r)=\frac 1 r$. In particular, 
$$\mult_E \bfix(K_{X_r})=\frac 1 r.$$
Thus, since $r$ is prime,  Lemma \ref{l_b} implies that if $q$ is a positive integer and  $R(X,qK_X)$ is  generated in degree $m$, then either $q\ge r$ or $m\ge r$.
\end{example}

\begin{proof}[Proof of Corollary \ref{c_main}]
It follows immediately from Theorem \ref{t_ld_threefold} and Lemma \ref{l_b}.
\end{proof}

\bibliographystyle{halpha}
\bibliography{Library}

\end{document}